\documentclass[11pt,a4paper,reqno]{amsart}

\usepackage{geometry}
\usepackage{amsmath}
\usepackage{amsfonts}
\usepackage{amssymb}
\usepackage{amsthm}
\usepackage{amscd}
\usepackage{stmaryrd,mathrsfs}
\usepackage{enumerate}
\usepackage{upgreek}
\usepackage{fancyhdr}
\usepackage{hyperref}

\newcommand{\Section}[1]{\section{#1} \setcounter{equation}{0}}

\newtheorem{theorem}{Theorem}[section]
\newtheorem{corollary}[theorem]{Corollary}
\newtheorem{lemma}[theorem]{Lemma}
\DeclareMathOperator{\card}{card}

\makeatletter
\def\imod#1{\allowbreak\mkern6mu({\operator@font mod}\,\,#1)}
\makeatother

\begin{document}
\large
\title{Squarefree density of polynomials}
\author{J. M. Kowalski}
\author{R. C. Vaughan}

%\pagestyle{fancy}
%\setlength{\headheight}{14.0pt}
%\fancyhf{}
%\rhead{,\thepage}
%\lhead{Kowalski, Vaughan}

\address{\noindent JMK: Dept. of Mathematics, Penn. State University, University Park, PA 16802, USA.}
\email{ jmk672@psu.edu}
\address{\noindent RCV: Dept. of Mathematics, Penn. State University, University Park, PA 16802, USA.}
\email{rcv4@psu.edu}

%\thanks{AMS Mathematics Subject Classification 2000: 11P82, 11P55.}

\begin{abstract}
\noindent This paper is concerned with squarefree values of polynomials
\[
\mathcal P(\mathbf x) \in\mathbb Z[x_1,\ldots,x_s].
\]
where we suppose that for each $j\le s$ we have $|x_j|\le P_j$.  Then we define
\[
N_{\mathcal P} (\mathbf P) = \sum_{\substack{\mathbf x\\
|x_j|\le P_j\\
\mathcal P(\mathbf x)\not=0}} \mu\big(|\mathcal P(\mathbf x)|\big)^2
\]
and we are interested in its behaviour when  $\min_jP_j\rightarrow\infty$, and the extent to which this can be approximated by
\[
N_{\mathcal P} (\mathbf P) \sim 2^sP_1\ldots P_s\mathfrak S_{\mathcal P}
\]
where
\[
\mathfrak S_{\mathcal P} = \prod_p \left(
1-\frac{\rho_{\mathcal P}(p^2)}{p^{2s}}
\right)
\]
and
\[
\rho_{\mathcal P}(d) = \card\{\mathbf x\in \mathbb Z_d^s: \mathcal P(\mathbf x)\equiv 0\imod{d}\}.
\]
We establish this is a number of new cases, and in particular show that if $s\ge 2$ and $\mathfrak S_{\mathcal P}=0$, then
\[
N_{\mathcal P} (\mathbf P) =o(P_1\ldots P_s).
\]
as $\min_jP_j\rightarrow\infty$.
\end{abstract}

\maketitle
%\today
%\vspace{-20 pt}

\Section{Introduction and Statement of Results}
\label{sec:one}
\noindent The work described in this paper has its origins in a talk given by Manjul Bhargava at the 60th birthday conference for Krishna Alladi at the University of Florida in March 2016, the main contents of which are in Bhargava \cite{MB14} and Bhargava, Shankar and Wang \cite{BSW}.  This was concerned with squarefree values of integral forms of degree $d$ in $s\ge 2$ variables $\mathbf x=(x_1,\ldots,x_s)$,
\begin{equation}
\label{eq:one1}
\mathcal F(\mathbf x) = \sum_{\substack{i_1,\ldots,i_d\\
i_1\le \ldots\le i_d\le s}} c_{i_1\ldots  i_d} x_{i_1}\ldots x_{i_d}
\end{equation}
in various special cases.
\par
The main result of this paper, Theorem \ref{thm:one1} below, was intended as part of an attack on the cubic case.  However very recently this was resolved by Lapkova and Xiao \cite{LX21} by a different method.  Nevertheless Theorem \ref{thm:one1} has many other uses and we give some examples here.
\begin{theorem}
\label{thm:one1}
Suppose that $s\ge 2$ and $\mathcal C(\mathbf x)$ is an integral cubic form not of the shape $a(b_1x_1+\cdots+b_sx_s)^3$ where $a,b_1,\ldots,b_s\in\mathbb Q$, and let $M(P)$ denote the number of solutions of $\mathcal C(\mathbf x) =\mathcal C(\mathbf X)$ with $\mathbf x$, $\mathbf X\in[-P,P]^s$.  Then
\[
M(P)\ll P^{2s-2+\varepsilon}.
\]
\end{theorem}
Apart possibly from the $\varepsilon$, this is best possible, as can be seen with the example
\[
cx_1^3+c'(x_2+\cdots+x_s)^3.
\]
\begin{corollary}
\label{thm:one2}
Suppose that $s\ge 3$, $k=3$ or $4$, $c\in\mathbb Z\setminus\{0\}$ and the integral cubic form $\mathcal C^*(x_2,\ldots,x_s)$ is not of the form $a(b_2x_2+\cdots+b_sx_s)^3$ where $a,b_2,\ldots,b_s\in\mathbb Q$, let $Q=P^{3/k}$, and let $L(P)$ denote the number of solutions of
\[
cx_1^k+\mathcal C^*(x_2,\ldots,x_s) =cX_1^k + \mathcal C^*(X_2,\ldots,X_s)
\]
with $|x_1|\le Q$, $|X_1|\le Q$, $|x_j|\le P$, $|X_j|\le P$ $(2\le j\le s)$.  Then
\[
L(P)\ll P^{2s-3+\varepsilon}Q^{1/2}.
\]
\end{corollary}
\par
There is, of course, no reason to restrict the original question to forms.  Instead one can consider general integral polynomials
\[
\mathcal P(\mathbf x) \in\mathbb Z[x_1,\ldots,x_s].
\]
where we suppose that for each $j\le s$ we have $|x_j|\le P_j$.  We extend the definition of the M\"obius function by taking $\mu(0)=0$.  Then we define
\begin{equation}
\label{eq:one2}
N_{\mathcal P} (\mathbf P) = \sum_{\substack{\mathbf x\\
|x_j|\le P_j}} \mu\big(|\mathcal P(\mathbf x)|\big)^2
\end{equation}
and we are interested in its behaviour when  $\min_jP_j\rightarrow\infty$, and the extent to which this can be approximated by
\begin{equation}
\label{eq:one3}
N_{\mathcal P} (\mathbf P) \sim 2^sP_1\ldots P_s\mathfrak S_{\mathcal P}
\end{equation}
where
\begin{equation}
\label{eq:one4}
\mathfrak S_{\mathcal P} = \prod_p \left(
1-\frac{\rho_{\mathcal P}(p^2)}{p^{2s}}
\right)
\end{equation}
and
\begin{equation}
\label{eq:one5}
\rho_{\mathcal P}(d) = \card\{\mathbf x\in \mathbb Z_d^s: \mathcal P(\mathbf x)\equiv 0\imod{d}\}.
\end{equation}
Note that
\[
\prod_{p\le n} \left(
1-\frac{\rho_{\mathcal P}(p^2)}{p^{2s}}
\right)
\]
is a non-negative decreasing sequence so it converges as $n\rightarrow\infty$ to a non-negative limit, although when the limit is $0$ it is usual to describe such a product as diverging!
\par
It seems that (\ref{eq:one3}) should hold in all cases.  Thus if $\mathcal P$ is such that it has a shortage of squarefree values, then we expect that
\begin{equation}
\label{eq:one6}
\mathfrak S_{\mathcal P}=0.
\end{equation}
In this case (\ref{eq:one3}) is easy to prove, although we have not seen it in the extant literature.
\begin{theorem}
\label{thm:one3}
Suppose that $s\ge 2$ and $\mathcal P\in\mathbb Z[x_1,\ldots,x_s]$ is an integral polynomial.  If $\mathfrak S_{\mathcal P}=0$, then
\[
N_{\mathcal P} (\mathbf P) =o(P_1\ldots P_s).
\]
as $\min_jP_j\rightarrow\infty$.
\end{theorem}\par
It is sometimes asserted that for large $s$ such results should follow by an application of the Hardy-Littlewood method.  However, as far as we are aware, there is nothing in the extant literature for general $\mathcal P$.  Moreover it is not entirely clear how methods currently available for showing the existence of non-trivial representations of $0$ by even cubic forms can be adapted to this situation.  First of all such methods take a small region about a non-trivial real zero of the form and expand it homothetically, but in some sense we need to approximate the number of solutions to
\[
\mathcal P(\mathbf x) =n
\]
for most $n$ and there is no clear way of adapting the method to cover the $s$-dimensional boxes considered here.  Second of all, those methods either require the form to satisfy a condition such as being non-singular in order to show that the concomitant exponential sum
\[
S(\alpha) = \sum_{|\mathbf x|\le P} e\big(\alpha\mathcal P(\mathbf x)\big)
\]
is relatively small on the minor arcs, or have to use an alternative argument when $S(\alpha)$ is not always small on the minor arcs which only ensures the existence of solutions rather than gives a good approximation for their number.
\par
We apply Corollary \ref{thm:one2} above to establish the following.
\begin{theorem}
\label{thm:one4}
Suppose that $s\ge 3$, $k=3$ or $4$, $c\in\mathbb Z\setminus\{0\}$ and the integral cubic form $\mathcal C^*(x_2,\ldots,x_s)$ is not of the shape $a(b_2x_2+\cdots+b_sx_s)^3$ where $a,b_1,\ldots,b_s\in\mathbb Q$.  Let
\begin{equation}
\label{eq:one7}
\mathcal P(\mathbf x) = cx_1^k  + \mathcal C^*(x_2,\ldots,x_s).
\end{equation}
with $c\not=0$.  Suppose $P$ is large and let $Q=P_1=P^{3/k}$, $P_j=P$ $(2\le j\le s)$, and
\[
N_{\mathcal P} (\mathbf P) = \sum_{|x_1|\le Q} \sum_{\substack{x_2,\ldots,x_s\\
|x_j|\le P}} \mu\big(|\mathcal P(\mathbf x)|\big)^2.
\]
Then, as $P\rightarrow\infty$,
\[
N_{\mathcal P} (\mathbf P) = 2^sQP^{s-1}\mathfrak S_{\mathcal P} + E_k
\]
where
\[
E_3\ll P^{s-\frac1{8s-2}+\varepsilon}
\]
and
\[
E_4\ll \frac{QP^{s-1}}{(\log P)\log\log P}.
\]
\end{theorem}
We remark that the case
\[
\mathcal C^*(x_2,\ldots,x_s) = a(b_2x_2+\cdots+b_sx_s)^3
\]
can be dealt with by the methods of Filaseta \cite{MF94} and Sanjaya and Wang \cite{SW23}.
\par
There is a useful survey on the squarefree density for this and various concomitant questions by Tsvetkov \cite{RT19}.
\par
Generally formul\ae\, involving $\varepsilon$ will hold for every $\varepsilon>0$ and the stated ranges of the other variables, albeit any implicit constant may well depend on $\varepsilon$.  Unless stated otherwise implicit constants will not depend on the other variables.  The authors would like to thank the referee for carefully reading the manuscript.

\Section{Lemmata}
\label{sec:two}

\noindent We begin with a pair of results from the appendix of Estermann \cite{TE}.

\begin{lemma}
\label{lem:two1}
Suppose that $a$, $b$, $n$ are positive integers and let $Q(n;a,b)$ denote the number of solutions of $ax^2+by^2=n$ in positive integers $x$ and $y$.  Then
$$Q(n;a,b)\le 2d(n).$$
\end{lemma}

\begin{lemma}
\label{lem:two2}
Suppose that $a$, $b$, $m$, $n$ are positive integers and let $R(n;a,b)$ denote the number of solutions of $ax^2-by^2=n$ in positive integers $x$ and $y$ with $ax^2\le m$.  Then
$$R(n;a,b)\le 2d(n)(1+\log m).$$
\end{lemma}

The following is probably in the literature, but we don't know where.
\begin{lemma}
\label{lem:two3}
Suppose that $n\in\mathbb N$, $b\in\mathbb Z$ and $S(P;n,b)$ denotes the number of integers $y\in[-P,P]$ such that $ny+b$ is a perfect square.  Let $k^2$ be the largest square dividing $(n,b)$.  Then
$$S(P;n,b) \ll n^{\varepsilon}\left(
1+ k\sqrt{\frac{P}{n}}
\right).$$
\end{lemma}
It is clear from the example $n=1$, $b=0$ that this is essentially best possible.  Gallagher’s larger sieve gives a somewhat sharper bound but the above is adequate for our purposes.
\begin{proof}
Write $(n,b)=k^2l$ so that $l$ is square free, and let $n_1=n/(n,b)$ and $b_1=b/(n,b)$, so that $(n_1,b_1)=1$.  If $x$ satisfies $x^2=ny+b$ then $kl|x$.  We wish to bound the number of pairs $z$, $y$ with $z\ge 0$, $y\in[-P,P]$ and $lz^2=n_1y+b_1$.  Since $(n_1,b_1)=1$ we have $(l,n_1)=1$.  Let $\mathcal R$ denote the set of residue classes $r$ modulo $n_1$ such that $lr^2\equiv b_1\imod{n_1}$.  As $(lb_1,n_1)=1$ we have $\mathrm{card}\,\mathcal R\ll n_1^{\varepsilon}$.  Let $r\in\mathcal R$.  Then it suffices to bound the number of solutions with $z\equiv r\imod{n_1}$.  Let $z_0$ be the least such solution and $y_0$ the corresponding value of $y$.  Thus for any other solution $z$ we have $z=z_0 + n_1 v$ where $v\ge 0$.  Hence $lv(2z_0+n_1v) = y-y_0$ and so $v|2z_0+n_1v|\le 2P/l$.  The number of possible $v\le (P/ln_1)^{1/2}$ is
\begin{equation}
\label{eq:two1}
\ll 1+\sqrt{\frac{P}{ln_1}}
\end{equation}
and when $v>(P/ln_1)^{1/2}$ we have $|2z_0+n_1v|\le 2(Pn_1/l)^{1/2}$ and again the number of such $v$ satisfies (\ref{eq:two1}).
\end{proof}

We apply these results to obtain bounds for a general quadratic polynomial in two variables.  In the lemma below the results depend on whether the three quantities $a^2+c^2$, $\Delta$ and $\Theta$ are zero or not.  As discussed in the remark below the lemma, the lemma takes care of all eight possibilities.

\begin{lemma}
\label{lem:two4}
Let $P\ge 2$ and $Q\ge2$, suppose that $a$, $b$, $c$, $d$, $e$ and $f$ are integers in $[-Q,Q]$ and let $\Delta=4ac-b^2$ and $\Theta = 4acf+ebd-ae^2-cd^2-fb^2$.  Let N(P,Q) denote the number of solutions of
$$ax^2+bxy+cy^2+dx+ey+f=0$$
in integers $x$, $y$ with $|x|\le P$ and $|y|\le P$.  There are various cases.\par
\noindent (i) If $(a^2+c^2)\Delta\Theta\not=0$, then $N(P,Q) \ll (PQ)^{\varepsilon}$.
\par
\noindent (ii) If $(a^2+c^2)\Delta\not=0$, $\Theta=0$ and $-\Delta$ is not a perfect square, then $N(P,Q) \ll 1$.
\par
\noindent (iii) If $a=c=0$ and $\Delta\Theta\not=0$, then $N(P,Q)\ll (PQ)^{\varepsilon}$.
\par
\noindent (iv) If $(a^2+c^2)\Delta\not=0$, $\Theta=0$ and $-\Delta$ is a perfect square, then $N(P,Q) \ll P$.
\par
\noindent (v) If $a\not=0$, $2ae-bd=\Delta=0$ and $d^2-4af$ is not a perfect square, or if $c\not=0$, $2cd-be=\Delta=0$ and $e^2-4cf$ is not a perfect square, then $N(P,Q) =0$.
\par
\noindent (vi) If $a\not=0$, $l=2ae-bd\not=0$, $\Delta=0$ and $k^2$ is the largest square dividing $|l|$, or if $c\not=0$, $l=2cd-be\not=0$, $\Delta=0$ and $k^2$ is the largest square dividing $|l|$, then
$$N(P,Q) \ll Q^{\varepsilon} \left(
1+k\sqrt{\frac{P}{|l|}}
\right).$$
\par
\noindent (vii) If $a\not=0$, $2ae-bd=\Delta=0$ and $d^2-4af$ is a perfect square, or if $c\not=0$, $2cd-be=\Delta=0$ and $e^2-4cf$ is a perfect square, then $N(P,Q) \ll P$.
\par
\noindent (viii) If $a=c=\Theta=0$ and $\Delta\not=0$, then $N(P,Q)\ll P$.
\par
For completeness we also state,
\par
\noindent (ix) in all other cases $N(P,Q)\ll P^2$.
\end{lemma}

Apropos our previous remark we see that (i) takes care of the case $(a^2+c^2)\Delta\Theta\not=0$, (ii) and (iv) collectively take care of the case $(a^2+c^2)\Delta\not=0$, $\Theta=0$ and (iii) takes care of the case $\Delta\Theta\not=0$, $a^2+c^2=0$.  When $(a^2+c^2)\not=0$ and $\Delta=0$ there are various subcases and these are taken care of by (v), (vi), (vii).  Since these subcases are indifferent to the value of $\Theta$ they deal with both the case $(a^2+c^2)\Theta\not=0$, $\Delta=0$ and the case $(a^2+c^2)\not=0$ and $\Delta=\Theta=0$.  The case $\Delta\not=0$, $a=c=\Theta=0$ is taken care of in (viii).  Finally if $a^2+c^2=\Delta=0$, then automatically $\Theta=0$ so the remaining cases are dealt with by (ix).  It is not necessary for our purposes but in this last situation when $d$ or $e$ is non-zero the bound could be replaced by $2P+1$.

\begin{proof} If $a=c=0$, then $\Delta = -b^2$ and $\Theta = b(ed-fb)$.  Hence if $\Delta\not=0$, so that $b\not=0$, multiplication by $b$ gives
$$(bx+e)(by+d) + bf-ed=0$$
and this has $\ll (PQ)^{\varepsilon}$ solutions when $\Theta\not=0$ and $\ll P$ solutions when $\Theta=0$.  This deals with cases (iii) and (viii) and so, apart from the trivial case (xi), we can suppose that $a^2+c^2\not=0$.  Thus, without loss of generality, we suppose that $a\not=0$.  Completing the square gives the equation
$$(2ax+by+d)^2+\Delta y^2 +2(2ae-bd)y+4af-d^2=0.$$
\par Suppose $\Delta=0$.  Then we are asking that $2(bd-2ae)y +d^2-4af$ be a perfect square.  If $2ae-bd=0$ and $d^2-4af$ is not a perfect square, then we have no solutions and we are in case (v).  If $bd\not=2ae$, then we are in case (vi) and this follows from Lemma \ref{lem:two3}.  If $bd=2ae$  and $d^2-4af$ is a perfect square, then $x$ is determined by $y$ and we get case (vii).
\par
Thus we may suppose that $\Delta\not=0$.  Completing the square once more gives
\begin{equation}
\label{eq:two2}
\Delta(2ax+by+d)^2 +(\Delta y + 2ae-bd)^2 + 4a\Theta=0.
\end{equation}
If $\Theta=0$ and $-\Delta$ is not a perfect square, then solutions are only possible with
$$2ax+by+d=\Delta y +2ae-bd=0$$
and so $y$ is determined by the second equation and then $x$ is determined by the first one. This gives case (ii).\par
If $\Theta=0$ and $-\Delta$ is a non-zero perfect square, say $m^2$, then the equation becomes
$$(2amx+mby+md)^2-(m^2 y+2ae-bd)^2=0.$$
Taking any choice for $y$ gives $\ll 1$ choices for $x$ and gives case (iv).
\par
Finally we may suppose  that $(a^2+c^2)\Delta\Theta\not=0$, case (i).   We can write the equation (\ref{eq:two2}) in the form
$$\Delta X^2 +Y^2= -4a\Theta.$$
If $\Delta>0$ and $-4a\Theta\le 0$, then there is at most one solution in integers $X$ and $Y$.  If $\Delta>0$ and $-4a\Theta>0$, then we can appeal to Lemma \ref{lem:two1} and we see that the number of solutions in $X$ and $Y$ is $\ll Q^{\varepsilon}$.  For each such pair $X$, $Y$ there are $\ll 1$ pairs $x$ and $y$ with $\Delta y+2ae-bd=Y$ and $2ax+by+d=X$.  If $\Delta<0$ and $-4a\Theta>0$, then we write the equation in the form
$$Y^2 - (-\Delta)X^2 = -4a\Theta$$
and use Lemma \ref{lem:two2} instead.  On the other hand if $4a\Theta>0$, then we can rewrite the equation as
$$(-\Delta) X^2 -Y^2 = 4a\Theta$$
and proceed in the same way.
\end{proof}

We will need a bound for $\rho_{\mathcal P}(d^2)$ when $d$ is square free.

\begin{lemma}
\label{lem:two5}
Let $k=3$ or $4$ and $\mathcal P$ be as in (\ref{eq:one7}).  Then
\[
\rho_{\mathcal P}(p^2)\ll p^{2s-2},
\]
and, for squarefree $d$,
\[
\rho_{\mathcal P}(d^2)\ll d^{2s-2+\varepsilon}.
\]
\end{lemma}
\begin{proof}
Since $\rho$ is a multiplicative function it suffices to show the first bound.  Furthermore we may suppose that $p>3$ and larger than any of the coefficients of $\mathcal P$.  We are counting the number of solutions of
\begin{equation}
\label{eq:two3}
cx_1^k+\mathcal C^*(x_2,\ldots,x_s) \equiv 0 \imod{p^2}
\end{equation}
Consider first the number of solutions of
\begin{equation}
\label{eq:two4}
cx_1^k+\mathcal C^*(x_2,\ldots,x_s) \equiv 0 \imod p.
\end{equation}
Let $Z$ denote the number of solutions of $C^*(x_2,\ldots,x_s) \equiv 0 \imod p$.  By Lemma 3.1 of \S4.3 of Schmidt \cite{WS04}, $Z\le 3p^{s-2}$.  Thus the number of solutions of (\ref{eq:two3}) with $x_1\equiv 0\imod p$ is $\ll 3p^s.p^{s-2}=3p^{2s-2}$.  The number of solutions of (\ref{eq:two4}) with $x_1\not=0$ is at most $kp^{s-1}$.  By the initial assumption on $p$ each solution of (\ref{eq:two4}) with $x_1\not\equiv 0\imod p$ is non-singular and so lifts to at most $p^{s-1}$ solutions modulo $p^2$.  Thus the number of solutions to (\ref{eq:two3}) with $x_1\not\equiv 0\imod p$ is $\le 3p^{2s-2}$.
\end{proof}

\Section{Proof of Theorem \ref{thm:one3}}
\label{sec:three}

\noindent If there should be a prime $p$ with $\rho_{\mathcal P}(p^2)=p^{2s}$, then $\mathcal P$ has $p^2$ as a fixed divisor, $\mathfrak S_{\mathcal P}$ is perforce $0$ and $N_{\mathcal P}(P)=0$.  Hence we can assume that, for every prime $p$,
\begin{equation}
\label{eq:three1}
\rho_{\mathcal P}(p^2)<p^{2s}.
\end{equation}
We have $\mathfrak S_{\mathcal P}=0$.  Thus, by (\ref{eq:one6}), given $\eta>0$ when $\xi>\xi_0(\eta)$ for some $\xi_0(\eta)$ we have
\begin{equation}
\label{eq:three2}
0< \prod_{p\le \xi} \left(
1-\frac{\rho_{\mathcal P}(p^2)}{p^{2s}}
\right)<\eta.
\end{equation}
Let
\begin{equation}
\label{eq:three3}
D=\prod_{p\le \xi}p.
\end{equation}
Then, by (\ref{eq:one2}), we have
\[
N_{\mathcal P}(\mathbf P) \le \sum_{\substack{\mathbf x\\
|x_j|\le P_j}} \sum_{m^2|(D^2,\mathcal P(\mathbf x))}\mu(m) =\sum_{m|D} \mu(m) \sum_{\substack{\mathbf x\\
|x_j|\le P_j\\
m^2|\mathcal P(\mathbf x)}} 1.
\]
Then by apportioning the $x_j$ into residue classes modulo $m^2$ we find that the inner sum is
\[
\rho_{\mathcal P}(m^2) \left(
\frac{2^sP_1\ldots P_s}{m^{2s}} +O\big(m^{2-2s}P_1\ldots P_s(\min_j P_j)^{-1} +1\big)
\right).
\]
The trivial bound $\rho_{\mathcal P}(m^2)\le m^{2s}$ then gives the approximation
\[
2^sP_1\ldots P_s\frac{\rho_{\mathcal P}(m^2)}{m^{2s}} + O\big( m^2P_1\ldots P_s(\min_j P_j)^{-1}+m^{2s}\big).
\]
Hence
\[
0\le N_{\mathcal P}(\mathbf P) \le 2^sP_1\ldots P_s\sum_{m|D} \rho_{\mathcal P}(m^2)\frac{\mu(m)}{m^{2s}} + O_{\xi}\big(P_1\ldots P_s(\min_j P_j)^{-1}\big).
\]
and so, by (\ref{eq:three2}),
\[
0\le \limsup_{{\min_j}P_j\rightarrow\infty} \frac{N_{\mathcal P}(\mathbf P)}{2^sP_1\ldots P_s} \le \sum_{m|D} \rho_{\mathcal P}(m^2)\frac{\mu(m)}{m^{2s}} = \prod_{p\le \xi} \left(
1-\frac{\rho_{\mathcal P}(p^2)}{p^{2s}}
\right)<\eta.
\]

\Section{Proof of Theorem \ref{thm:one1}}
\label{sec:four}

\noindent The proof is inductive.  The case $s=2$ can be deduced from results in the literature.  When the polynomial $\mathcal C(x,1)$ is irreducible over $\mathbb Q$, then, by Bombieri and Schmidt \cite{BS87}, when $n\in\mathbb Z\setminus\{0\}$ the number of solutions of the Thue equation
\begin{equation}
\label{eq:four1}
\mathcal C(x,y)=n
\end{equation}
is $\ll |n|^{\varepsilon}$.  When the polynomial is reducible over $\mathbb Q$, since both variables occur explicitly and it is not of the form $a(b_1x_1+b_2x_2)^3$, then (\ref{eq:four1}) is of the form
\[
(Ax+By)(Cx^2+Dxy+Ey^2)=Fn
\]
for integral $A,..,,F$ with $F\not=0$ and $Ax+By$ not a factor of $Cx^2+Dxy+Ey^2$.  Then the factors are complementary divisors $d$, $d'$ of $Fn$.  Without loss of generality we may suppose that $A\not=0$.  Then $x=(d-By)/A$ and so by substitution
\[
C(d-By)^2 + DA(d-By)+EA^2y^2=A^2d'.
\]
The coefficient of $y^2$ here is $CB^2-DAB+EA^2\not =0$ since $Ax+By$ is not a factor of $Cx^2+Dxy+Ey^2$.  Thus there are at most $2$ choices for $y$ and then $x$ is fixed, so the number of solutions of $\mathcal P(x,y)=\mathcal P(X,Y)\not=0$ is $\ll P^{\varepsilon}$ times the number of choices for $X,Y$.  The remaining situation $\mathcal P(x,y)=\mathcal P(X,Y)=0$ clearly has $\ll P^2$ solutions.
\par
Now suppose that $s>2$.  It is useful to first transform the form so that at least two of the variables, for example $x_1$ and $x_2$, have non-zero $x_1^3$ and $x_2^3$ terms.  Suppose on the contrary that $c_{111}=0$.  If there is a $j$ so that $c_{11j}\not=0$ or $c_{1jj}\not=0$ or $c_{jjj}$, then replace $x_j$ by $x_j+\lambda x_1$ where $\lambda$ is a non-zero integer at our disposal.  Then $x_1^3$ will now occur as $(3c_{11j}\lambda + 3c_{1jj}\lambda^2 +c_{jjj}\lambda^3)x_1^3$.  This has a non-zero coefficient for a sufficiently large $\lambda$.  If $c_{11j}=c_{1jj}=c_{jjj}=0$ for every $j$, then since $x_1$ occurs explicitly there will be $1<j<k$ such that $c_{1jk}\not=0$.  Now replace $x_j$ and $x_k$ by $x_j+\lambda x_1$ and $x_k+\mu x_1$ respectively where now $\lambda$ and $\mu$ are at our disposal.  Then $x_1^3$ occurs as
\[
(c_{1jk}\lambda\mu + c_{jjk}\lambda^2\mu + c_{jkk}\lambda\mu^2)x_1^3.
\]
If $c_{jkk}\not =0$ then one can take $\lambda=1$ and $\mu$ sufficiently large, and if $c_{jkk}=0$ and  $c_{jjk}\not=0$, then one can take $\mu=1$ and $\lambda$ large..  Finally if $c_{jjk}=c_{jkk}=0$ then one can take $\lambda=\mu=1$.  The transformations in each case are unimodular, so invertible and the new forms therefore represent the same numbers as the old ones, and {\it vice versa}.  The region of interest becomes a parallelepiped, but can certainly be accommodated in $[-cP,cP]^s$ for a suitable constant $c$.  The process can be repeated to ensure that at least two variables $x_j$ appear explicitly as $x_j^3$.  Note that the transformation $x_k$ to $x_k+\lambda_k x_j$ $(k\not=j)$ does not alter the coefficient of $x_k^3$
\par
Henceforward we can suppose that two of our variables, say $x$ and $y$, are such that $x^3$ and $y^3$ occur explicitly.  Denote the remaining variables by $\mathbf z$.  Thus
\begin{multline*}
\mathcal C(\mathbf x) = Ax^3 + Bx^2y + Cxy^2 + Dy^3 + \\
x^2\mathcal L_1(\mathbf z) + xy \mathcal L_2(\mathbf z) + y^2\mathcal L_3(\mathbf z) + x\mathcal Q_1(\mathbf z) + y\mathcal Q_2(\mathbf z) + \mathcal C_1(\mathbf z)
\end{multline*}
where $\mathcal L_j$, $\mathcal Q_j$ and $\mathcal C_1$ are linear, quadratic and cubic forms respectively.
\par
If the $\mathcal L_j$, $\mathcal Q_j$ are all identically $0$, then $\mathcal C(\mathbf x) = \mathcal C_0(x,y) + \mathcal C_1(\mathbf z)$, where $\mathcal C_0$ is a binary cubic form.  Moreover the number of solutions of $\mathcal C_0(x,y) - \mathcal C_0(X,Y)=m$ is bounded by the number with $m=0$, so we can appeal to the case $s=2$.
\par
Thus we can assume that $AD\not=0$ and not all of $\mathcal L_j$, $\mathcal Q_j$ are identically $0$.  Let $R(n)$ denote the number of solutions of $\mathcal C(\mathbf x)= n$ with $\mathbf x\in[-P,P]^s$, where here $\mathbf x$ is a shorthand for $(x,y,z_1,\ldots,z_{s-2})$.  We sort the solutions according to $\mathbf z$.  Thus
$$R(n)=\sum_{\mathbf z\in[-P,P]^{s-2}} R(n,\mathbf z)$$
where $R(n,\mathbf z)$ is the number of solutions $x,y$ of
\begin{multline*}
Ax^3 + Bx^2y + Cxy^2 + Dy^3 + \\
x^2\mathcal L_1(\mathbf z) + xy \mathcal L_2(\mathbf z) + y^2\mathcal L_3(\mathbf z) + x\mathcal Q_1(\mathbf z) + y\mathcal Q_2(\mathbf z) + \mathcal C_1(\mathbf z) = n
\end{multline*}
with $x$, $y\in[-P,P]$.  The object of the theorem, namely the number $M(P)$ of solutions of $\mathcal C(\mathbf x)= \mathcal C(\mathbf X)$ can be written as
$$\sum_{n\in\mathbb Z} R(n)^2.$$
By Cauchy's inequality this is at most
$$(2P+1)^{s-2} \sum_{\mathbf z \in [-P,P]^{2s-2}} \sum_{n\in\mathbb Z} R(n,\mathbf z)^2.$$
Thus the equation to be considered becomes
\begin{align*}
Ax^3 + Bx^2y + Cxy^2 + Dy^3 + x^2\mathcal L_1(\mathbf z) + xy \mathcal L_2(\mathbf z) &+ y^2\mathcal L_3(\mathbf z) + x\mathcal Q_1(\mathbf z) + y\mathcal Q_2(\mathbf z) \\
=AX^3 + BX^2Y + CXY^2 + DY^3 +
X^2\mathcal L_1(\mathbf z) + &XY \mathcal L_2(\mathbf z) + Y^2\mathcal L_3(\mathbf z) + X\mathcal Q_1(\mathbf z)\\
& + Y\mathcal Q_2(\mathbf z)
\end{align*}
This has $s+2$ variables and we desire to show that there are $\ll P^{s+\varepsilon}$ solutions.
\par
We now define $g=X-x$, $h=Y-y$ so that the equation becomes
\begin{multline*}
Ag(3x^2+3gx+g^2) + B(x^2h+2gxy+2ghx+g^2y+g^2h) \\
+ C(2hxy+gy^2+h^2x+2ghy+gh^2) + Dh(3y^2+3hy+h^2) + \mathcal L_1(\mathbf z)(2xg+g^2) \\
 \mathcal L_2(\mathbf z)(xh+yg+gh) + \mathcal L_3(\mathbf z)(2yh+h^2) + g\mathcal Q_1(\mathbf z) + h\mathcal Q_2(\mathbf z) = 0.
\end{multline*}
We can rewrite this as
\begin{equation}
\label{eq:four2}
ax^2+bxy+ cy^2 + dx+ey +f=0
\end{equation}
where
\begin{equation}
\label{eq:four3}
a=3Ag+Bh,\quad b=2Bg+2Ch,\quad c=Cg+3Dh,
\end{equation}
\begin{equation}
\label{eq:four4}
d=3Ag^2+2Bgh+Ch^2+2\mathcal L_1(\mathbf z)g +\mathcal L_2(\mathbf z) h,
\end{equation}
\begin{equation}
\label{eq:four5}
e=Bg^2+2Cgh+3Dh^2+\mathcal L_2(\mathbf z)g + 2\mathcal L_3(\mathbf z) h
\end{equation}
and
\begin{multline}
\label{eq:four6}
f=Ag^3+Bg^2h+Cgh^2 +Dh^3 +\mathcal L_1(\mathbf z) g^2 +\mathcal L_2(\mathbf z)gh + \mathcal L_3(\mathbf z)h^2 \\
+ \mathcal Q_1(\mathbf z)g +\mathcal Q_2(\mathbf z)h.
\end{multline}
We emphasise that the coefficients $a$, through $e$ depend on $g$, $h$, $\mathbf z$, but not $x$, $y$.
\par
We will apply Lemma \ref{lem:two4} multiple times.  In the notation of Lemmma \ref{lem:two4}, if for any given $g$, $h$, $\mathbf z$ any of the cases (i), (ii), (iii) or (v) there hold, then we have a suitable bound for the number of corresponding solutions of (\ref{eq:four2}).
\par
Now suppose that case (iv) holds.  If $3AC-B^2\not=0$, then $\Theta=0$ ans $\Theta$ has a term $A(3AC-B^2)g^5$ and so $g$ is determined by $\mathbf z$ and $h$ alone.  Again we have a suitable bound.  If $3AC-B^2=0$, then $\Theta=0$ and $\Theta$ has a term $(9A^2D+2ABC-B^2)g^4h= A(9AD-BC)g^4h$ and so if $9AD-BC\not=0$, then $h=0$ or $g$ is determined by $h$ which again also gives a suitable bound.  Finally, if $3AC-B^2=BC-9AD=0$, then by considering the term in $h^5$ either we have the desired bound or $3AC-B^2=BC-9AD=C^2-3BD=0$.  Since $AD\not=0$ we have $BC\not=0$, and so $a=\frac{B}{C}(Bg+Ch)$, $b=2(Bg+Ch)$, $c=\frac{C}{B}(Bg+Ch)$.  But then $\Delta =0$ contrary to the assumption in this case.
\par
In the notation of Lemma \ref{lem:two4} the remaining situations are \par
{\bf 1.} $a^2+c^2\not=0$ and $\Delta=0$,
\par
{\bf 2.} $a=c=0=\Theta=0$ and $\Delta\not=0$.
\par
{\bf 3.} $a=c=\Delta=0$.
\par
{\bf In case 1.}, one of the cases (v), (vi) or (vii) of Lemma \ref{lem:two4} applies.  In case (v) there are no solutions, so we can forget that case.  In cases (vi) and (vii), $\Delta=0$ gives
\[
(B^2-3AC)g^2 +(BC-9AD)gh + (C^2-3BD)h^2=0
\]
If any of the coefficients $B^2-3AC$, $BC-9AD$, $C^2-3BD$ are non-zero, then $g$ or $h$ is $0$ or $g$ is determined by $h$ or {\it vice versa}. That leaves the situation, as in case (iv),
\[
B^2-3AC = BC-9AD = C^2-3BD=0.
\]
Since $AD\not=0$ we also have $BC\not=0$.  Now
\[
A=\frac{B^2}{3C},\quad D=\frac{C^2}{3B}
\]
and so
\begin{equation}
\label{eq:four7}
a=BC^{-1}(Bg+Ch),\quad c=CB^{-1}(Bg+Ch)
\end{equation}
and
\[
ax^2+bxy+cy^2 = B^{-1}C^{-1}(Bg+Ch)(Bx+Cy)^2.
\]
The above implies, {\it inter alia}, that $3C|B^2$ and $3B|C^2$.  Therefore, in what follows all of our expression are integer valued.
\par
Since $a\not=0$ or $c\not=0$ we have $Bg+Ch\not=0$.  We also have
$$d=C^{-1}(Bg+Ch)^2 + 2\mathcal L_1(\mathbf z)g +\mathcal L_2(\mathbf z) h,$$
$$e=B^{-1}(Bg+Ch)^2 + \mathcal L_2(\mathbf z)g +2\mathcal L_3(\mathbf z) h$$
and
\[
f = 3^{-1}B^{-1}C^{-1}(Bg+Ch)^3 + \mathcal L_1(\mathbf z) g^2 +\mathcal L_2(\mathbf z)gh + \mathcal L_3(\mathbf z)h^2 + \mathcal Q_1(\mathbf z)g +\mathcal Q_2(\mathbf z)h.
\]
\par
Then, by (\ref{eq:four3}) and (\ref{eq:four7}),
\begin{multline*}
2ae-bd = 2BC^{-1}(Bg+Ch)\left(
B^{-1}(Bg+Ch)^2 + \mathcal L_2(\mathbf z)g +2\mathcal L_3(\mathbf z) h
\right)\\ - 2(Bg+Ch)\left(
C^{-1}(Bg+Ch)^2 + 2\mathcal L_1(\mathbf z)g +\mathcal L_2(\mathbf z) h
\right).
\end{multline*}
so that
\begin{equation}
\label{eq:four8}
2ae-bd = 2(Bg+Ch) \left(
2BC^{-1}\mathcal L_3(\mathbf z)h + \left(
BgC^{-1}-h
\right) \mathcal L_2(\mathbf z) -2\mathcal L_1(\mathbf z) g
\right).
\end{equation}
Likewise
\begin{equation}
\label{eq:four9}
2cd-be = 2(Bg+Ch) \left(
2CHB^{-1}\mathcal L_1(\mathbf z)g + \left(
ChB^{-1}-g
\right) \mathcal L_2(\mathbf z) -2\mathcal L_3(\mathbf z) h
\right).
\end{equation}
In case (vi) consider first the possibility $a\not=0$, $2ae-bd\not=0$, $\Delta=0$.  let $l=2ae-bd$, so that $0<|l|\ll P^3$. Then $g$, $h$, $\mathbf z$ satisfy
\[
2(Bg+Ch) \big(
2B\mathcal L_3(\mathbf z)h + (Bg - Ch) \mathcal L_2(\mathbf z) -2C\mathcal L_1(\mathbf z) g
\big) = Cl.
\]
Given such $g$, $h$, $\mathbf z$, by Lemma \ref{lem:two4} (vi) the number of choices of $x$ and $y$ is
\[
\ll P^{\varepsilon} \left(
1+ k(P/|l|)^{\frac12}
\right)
\]
where $k^2|l$.  The total contribution from the $P^{\varepsilon}$ term is $\ll P^{s+\varepsilon}$, which is acceptable.  Thus it suffices to consider the contribution from the $P^{\varepsilon}k(P/|l|)^{\frac12}$ term.  Therefore we can presume that there is an $l$ with $|l|\in\mathbb N$ and a $u\in\mathbb Z\setminus\{0\}$ such that $2(Bg+Ch)=u$ and
\[
\big(B\mathcal L_2(\mathbf z)-2C\mathcal L_1(\mathbf z)\big)g +\big(2B\mathcal L_3(\mathbf z) -C\mathcal L_2(\mathbf z)\big) h = Cl/u.
\]
Given $g$, the first of these equations determines $h$.  In the second equation we have $Cl/u\not=0$ and the left side is a linear form in $\mathbf z$.  Thus at least one of the $z_j$ appears explicitly and so is determined by the other variables.  Hence, given $l$, the total number of possible $g$, $h$ and $\mathbf z$ is $\ll P^{s-2+\varepsilon}$.  Thus the total contribution is
\[
\ll P^{s-2+\varepsilon} \sum_{0<|l|\ll P^3}\sum_{k^2|l} k(P/|l|)^{\frac12} \ll P^{s-2+\varepsilon} \sum_{k\ll P^{3/2}} P^{\frac12} \sum_{0<j\ll P^3/k^2} j^{-\frac12} \ll P^{s+2\varepsilon}
\]
and we have an acceptable bound for the number of solutions in this case.
\par
Alternatively, if $a=0$ but $c\not=0$ a concomitant argument, where $2ae-bd$ is replaced by $2cd-be$ gives the desired bound.
\par
Now consider case (vii), and to begin with suppose that $a\not=0$, $2ae-bd=\Delta=0$ and $d^2-4af$ is a perfect square.  Then by (\ref{eq:four8}) we have
\[
(Bg+Ch) \left(
2BC^{-1}\mathcal L_3(\mathbf z)h + \left(
BgC^{-1}-h
\right) \mathcal L_2(\mathbf z) -2\mathcal L_1(\mathbf z) g
\right)=0.
\]
Since $a\not=0$, by (\ref{eq:four1}) we have $Bg+Ch\not=0$.  Thus
\[
2BC^{-1}\mathcal L_3(\mathbf z)h + \left(
BgC^{-1}-h
\right) \mathcal L_2(\mathbf z) -2\mathcal L_1(\mathbf z) g =0,
\]
so that
\[
\big(B\mathcal L_2(\mathbf z) -2C\mathcal L_1(\mathbf z)) g + (2B\mathcal L_3(\mathbf z) -C\mathcal L_2(\mathbf x)\big) h=0.
\]
If $B\mathcal L_2(\mathbf z) -2C\mathcal L_1(\mathbf z)$ or $2B\mathcal L_3(\mathbf z) -C\mathcal L_2(\mathbf x)$ is not identically $0$, then either $g$ is determined by $h$ and $\mathbf z$, or $h$ is by $g$ and $\mathbf z$, or there are $\ll P^{s-3}$ choices of $\mathbf z$ for which $B\mathcal L_2(\mathbf z) -2C\mathcal L_1(\mathbf z) =0$ or $2B\mathcal L_3(\mathbf z) -C\mathcal L_2(\mathbf x)=0$.  Thus we can then appeal to case (vii) of the lemma.  If $B\mathcal L_2(\mathbf z) -2C\mathcal L_1(\mathbf z)$ and $2B\mathcal L_3(\mathbf z) -C\mathcal L_2(\mathbf x)$ are identically $0$, then we have
\begin{align*}
0=&ax^2+bxy+cy^2 +dx+ey+f\\
=& (Bg+Ch)3^{-1}B^{-1}C^{-1}\left(
3(Bx+Cy)^2 + 3(Bx+Cy)(Bg+Ch) + (Bg+Ch)^2
\right) \\
&+ \mathcal L_1(\mathbf z) g^2 + \mathcal L_2(\mathbf z)gh + \mathcal L_3(\mathbf z)h^2 + \mathcal Q_1(\mathbf z)g + \mathcal Q_2(\mathbf z) h.
\end{align*}
Since $B\mathcal L_2(\mathbf z) -2C\mathcal L_1(\mathbf z)$ and $2B\mathcal L_3(\mathbf z) -C\mathcal L_2(\mathbf x)$ are identically $0$, we obtain
\[
G(3X^2 +3XG + G^2) + G^2B^{-2}\mathcal L_1(\mathbf z) + GB^{-1} \mathcal Q_1(\mathbf z) + hB^{-1} (B\mathcal Q_2(\mathbf z)- C\mathcal Q_1(\mathbf z)) = 0
\]
where $X=Bx+Cy$ and $G=Bg+Ch$.  If $B\mathcal Q_2(\mathbf z) - C\mathcal Q_1(\mathbf z) $ is not identically $0$, then that are at most $\ll P^{s-3}$ values of $\mathbf z$ for which it is $0$.  There are also $\ll P$ values of $g$ and $h$ for which $G=0$ or $h=0$.  Thus we can suppose that $B\mathcal Q_2(\mathbf z) - C\mathcal Q_1(\mathbf z) $ is not identically $0$ and $hG\not=0$.  But then $G|h(B\mathcal Q_2(\mathbf z) - C\mathcal Q_1(\mathbf z))$ and so given $h$ and $\mathbf z$ there are $\ll P^{\varepsilon}$ choices for $g$, so we have a suitable bound in case (vii).  The alternative case $c\not=0$, $2cd-be=\Delta=0$, $e^2-4cf$ a perfect square is similar.  That completes the analysis of 1.
\par
{\bf Case 2.} corresponds exactly to (viii) of Lemma \ref{lem:two4}.  Then by (\ref{eq:four3}) and the fact that $A\not=0$ we see that $g$ is determined by $h$, so that for any given $\mathbf z$, by (viii) of Lemma \ref{lem:two4}, the total number of choices of $x,y,g,h$ is $\ll P^2$ and that is sufficient.
\par
{\bf Case 3.}  This corresponds to (ix) of Lemma \ref{lem:two4}).  This  is similar to case (vii) but now $a=b=c=0$.  Then $g=-Bh/(3A)$, $h=-Cg/(3D)$.  Thus $B=0$ implies $g=0$ and so $h=0$ and this gives an adequate bound for the total number of such solutions.  Likewise $C=0$.  Hence we can suppose $BC\not=0$.  Moreover anyway $g$ is fixed by $h$ and {\it vice versa}.  It follows that
\[
g=\lambda h
\]
where
\[
\lambda = -B/(3A)=-3D/C=-C/B
\]
and
\begin{align*}
3Ag^2+2Bgh+Ch^2 &= 0,\\
Bg^2+2Cgh+3Dh^2 &=0, \\
Ag^3+Bg^2h+Cgh^2 +Dh^3 &=0,\\
\end{align*}
and so our equation reduces to
$$dx+ey +f$$
where
\[
d=2\mathcal L_1(\mathbf z)g +\mathcal L_2(\mathbf z) h,
\]
\[
e=\mathcal L_2(\mathbf z)g + 2\mathcal L_3(\mathbf z) h,
\]
\[
f=\mathcal L_1(\mathbf z) g^2 +\mathcal L_2(\mathbf z)gh + \mathcal L_3(\mathbf z)h^2\\
 + \mathcal Q_1(\mathbf z)g +\mathcal Q_2(\mathbf z)h.
\]
Recall that $g$ is fixed by $h$ and {\it vice versa}.  If $d\not=0$ or $e\not=0$, then $x$ or $y$ is fixed by the other variables.  Thus we can suppose that $d=e=0$ and hence $f=0$.  Thus
\begin{equation}
\label{eq:four10}
2\mathcal L_1(\mathbf z)\lambda +\mathcal L_2(\mathbf z)= \mathcal L_2(\mathbf z)\lambda + 2\mathcal L_3(\mathbf z) = 0
\end{equation}
and
\begin{equation}
\label{eq:four11}
\mathcal L_1(\mathbf z) \lambda^2 +\mathcal L_2(\mathbf z)\lambda + \mathcal L_3(\mathbf z)\\
+ \mathcal Q_1(\mathbf z)\lambda +\mathcal Q_2(\mathbf z)=0.
\end{equation}
Substituting from the first two equations into the third gives
\begin{equation}
\label{eq:four12}
\mathcal Q_1(\mathbf z)\lambda +\mathcal Q_2(\mathbf z)=0.
\end{equation}
If any one of the equations (\ref{eq:four10}) and (\ref{eq:four12}) does not hold identically, then it holds for at most $\ll P^{s-3}$ choices of $\mathbf z$ and we are done.  If they all hold identically, then $d=e=f=0$.  Returning to the original equation $\mathcal C(\mathbf x)=\mathcal C(\mathbf X)$ this gives $\mathcal C(x+\lambda h,y+h,\mathbf z) = \mathcal C(x,y,\mathbf z)$ identically for all $x$, $y$, $h$, $\mathbf z$.  Taking $h=-y$ this gives $\mathcal C(x,y,\mathbf z) = \mathcal C(x-\lambda y,0,\mathbf z)$ and then we can appeal to our inductive hypothesis since this reduces to the case $s-1$.

\Section{The proof of Corollary \ref{thm:one2}}
\label{sec:five}

\noindent We adapt the proof of Hua's Lemma (Lemma 2.5 of Vaughan \cite{RV97}).  Let
\[
f(\alpha)=\sum_{|x|\le Q} e(\alpha x^k)\text{ and }S(x)=\sum_{\substack{
x_2,\ldots x_s\\
|x_j|\le P
}} e\big(\alpha\mathcal C^*(x_2,\ldots,x_s)\big).
\]
Then, by the Cauchy-Schwarz inequality,
\begin{equation}
\label{eq:five1}
L(P) = \int_0^1 |f(c\alpha)S(\alpha)|^2 d\alpha\le \left(
\int_0^1 |S(\alpha)|^2 d\alpha
\right)^{1/2}\left(
 \int_0^1 |f(c\alpha)^2S(\alpha)|^2 d\alpha
\right)^{1/2}.
\end{equation}
By Weyl differencing and the Cauchy-Schwarz inequality
\[
|f(c\alpha)|^4\le \sum_{|h|\ll P^3} Q(h) e(\alpha h)
\]
where $Q(0)\ll Q^3$, $Q(h )\ll QP^{\varepsilon}\,(h\not=0)$.
\par
For $n\in\mathbb Z$, let $R(n)$ denote the number of solutions of $\mathcal C^*(x_2,\ldots,x_s)=n$ with $|x_j|\le P$.  Then, by Theorem \ref{thm:one1},
\[
\int_0^1 |S(\alpha)|^2 d\alpha =\sum_n R(n)^2\ll P^{2s-4+\varepsilon}.
\]
The theorem then follows from (\ref{eq:five1}) and the observation
\begin{align*}
\int_0^1 |f(c\alpha)^2S(\alpha)|^2 d\alpha &\le \sum_{n_1,n_2} R(n_1)R(n_2) Q(n_2-n_1) \\
&\ll Q^3 \sum_n R(n)^2 + QP^{\varepsilon} \left(\sum_nR(n)\right)^2\ll QP^{2s-2+2\varepsilon}.
\end{align*}

\Section{The proof of Theorem \ref{thm:one4}}
\label{sec:six}

\noindent For a given integer $n$, let $R(n)$ denote the number of choices of $\mathbf x$ with $|x_1|\le Q$, $|x_j|\le P$ $(2\le j\le s)$ and $\mathcal P(\mathbf x) = cx_1^k+\mathcal C^*(x_2,\cdots,x_s)=n$.  Then we have
\[
N_{\mathcal P}(\mathbf P) = \sum_{n\not=0} \mu(|n|)^2 R(n) = \sum_{d} \mu(d) \sum_{\substack{n\not=0\\
d^2|n}} R(n).
\]
Let $T\ge 1$ be at our disposal.  Then, by Cauchy's inequality and Corollary  \ref{thm:one2},
\begin{equation}
\label{eq:six1}
\sum_{d>T} \sum_{\substack{n\not=0\\
d^2|n}} R(n) \ll \left(
\sum_{T<d\ll P^{3/2}} \frac{P^3}{d^2}
\right)^{\frac12}\left(
\sum_{n} \sum_{d^2|n} R(n)^2
\right)^{\frac12} \ll T^{-\frac12} P^{s+\varepsilon}Q^{\frac14}.
\end{equation}
Hence
\begin{equation}
\label{eq:six2}
\sum_{d>T} \mu(d)\sum_{\substack{n\not=0 \\
d^2|n}} R(n) \ll QP^{s-1}PT^{-1/2}Q^{-3/4}.
\end{equation}
\par
Suppose first that $k=3$, so that $Q=P$.  For $d\le T$ we have
\[
\sum_{\substack{n \\
d^2|n}} R(n) = \left(
\frac{2P}{d^2}+O(1)
\right)^s \rho_{\mathcal P}(d^2) = \rho_{\mathcal P}(d^2)\frac{(2P)^s}{d^{2s}} + O\left(
\rho_{\mathcal P}(d^2)\frac{P^{s-1}}{d^{2s-2}} +\rho_{\mathcal P}(d^2)
\right).
\]
Hence, by Lemma \ref{lem:two5},
\[
\sum_{d\le T} \mu(d) \sum_{\substack{n\not=0 \\
d^2|n}} R(n) = (2P)^s\sum_{d=1}^{\infty} \mu(d)\frac{\rho_{\mathcal P}(d^2)}{d^{2s}} + E
\]
where $E\ll TR(0)+ P^s T^{\varepsilon-1} + T^{1+\varepsilon}P^{s-1} + T^{2s-1+\varepsilon}$.  Clearly $R(0)\ll P^{s-1}$, so by (\ref{eq:six2}),
\[
N_{\mathcal P}(\mathbf P) = (2P)^s \sum_{d=1}^{\infty} \mu(d)\frac{\rho_{\mathcal P}(d^2)}{d^{2s}} + E'
\]
where $E'\ll P^{s+1/4+\varepsilon}T^{-1/2} + P^s T^{\varepsilon-1} + T^{1+\varepsilon}P^{s-1} + T^{2s-1+\varepsilon}$.  Now $T = P^{\frac{4s+1}{8s-2}}$ gives
\[
E' \ll P^{s-\frac1{8s-2}+\varepsilon}.
\]
Also, by Lemma \ref{lem:two5}, and since $\rho_{\mathcal P}(d^2)$ is a multiplicative, we have
\[
\sum_{d=1}^{\infty} \mu(d)\frac{\rho_{\mathcal P}(d^2)}{d^{2s}} = \prod_{p} \left(
1-\frac{\rho_{\mathcal P}(p^2)}{p^{2s}}
\right).
\]
\par
Now suppose that $k=4$, so that $Q=P^{3/4}$.  This is more delicate.  Let
\[
X=\frac13\log P,\quad D=\prod_{p\le X} p
\]
and define
\[
N_0 = \sum_{n\not=0}\sum_{d^2|(D^2,n)}\mu(d) R(n),\quad N(U,V) = \sum_{U<p\le V} \sum_{\substack{n\\
p^2|n}} R(n)
\]
Then clearly every $n$ counted by $N_{\mathcal P}(\mathbf P)$ is counted by $N_0$ and those $n$ counted by $N_0$ but not by $N_{\mathcal P}(\mathbf P)$ will have a factor $p^2$ with $p>X$.  Thus
\[
N_0-N(P^{8/9},\infty)-N(|c|Q,P^{8/9})-N(P^{1/2},|c|Q)-N(X,P^{1/2})\le N_{\mathcal P}(\mathbf P)\le N_0.
\]
\par
The inequality (\ref{eq:six1}) gives
\[
N(P^{8/9},\infty) \ll QP^{s-4/9-9/16+\varepsilon}=QP^{s-\frac1{144}+\varepsilon}.
\]
In the inner sum in $N(|c|Q,P^{8/9})$, given $p$ with $|c|Q<p\le P^{8/9}$ we are counting the $\mathbf x$ with $|x_1|\le Q$, $|x_j|\le P$ $(2\le j\le s)$ and $p^2|cx_1^4+\mathcal C^*(x_2,\ldots,x_s)$.  When $j\ge 2$ write $x_j=u_j+pv_j$ where $1\le u_j\le p$, and so $|v_j|\ll P/p$.  When $x_1=0$ the number of choices for $u_2,\ldots,u_s$ is $\ll p^{s-2}$ and so the number of choices for $\mathbf x$ with $x_1=0$ is
\[
\ll P^{s-1}p^{-1}.
\]
When $x_1\not=0$ since $p>c|x_1|$ we have $p\nmid \mathcal C^*(u_2,\ldots,u_s)$.  Hence, by the Euler relation $u_2,\ldots,u_s$ is a non-singular point modulo $p$ of $\mathcal C^*$.  Moreover, given $x_1\not=0$ the number of choices for $u_2,\ldots, u_s$ is $\ll p^{s-2}$, and for each such, since $p^2>P$ the number of choices for $v_2,\ldots,v_s$ is $\ll (P/p)^{s-2}$.  Thus the total number of choices for $\mathbf x$ with $x_1\not=0$ is
\[
\ll Qp^{s-2}(P/p)^{s-2} = QP^{s-2}.
\]
Therefore
\[
N(|c|Q,P^{8/9}) \ll \sum_{|c|Q<p\le P^{8/9}} \big(P^{s-1}p^{-1}+QP^{s-2}\big)\ll QP^{s-10/9}.
\]
When $P^{1/2}<p\le |c|Q$ we proceed in the same way, but now the number of choices for $x_1\equiv 0\imod p$ is $\ll Q/p$.  Thus
\[
N(P^{1/2},|c|Q) \ll \sum_{P^{1/2}<p\le |c|Q} \big(QP^{s-1}p^{-2}+QP^{s-2}\big)\ll QP^{s-5/4}.
\]
For the sum $N(X,P^{1/2})$ for a given prime $p$ with $X<p\le P^{1/2}$ we divide the $\mathbf x$ into residue classes modulo $p^2$ and obtain the bound
\[
\sum_{X<p\le P^{1/2}} \left(
\frac{Q}{p^2} +1
\right)\frac{P^{s-1}}{p^{2s-2}}\rho_{\mathcal P}(p^2) \ll QP^{s-1} \sum_{X<p\le P^{1/2}} \frac{\rho_{\mathcal P}(p^2)}{p^{2s}} + P^{s-1}\sum_{X<p\le P^{1/2}} \frac{\rho_{\mathcal P}(p^2)}{p^{2s-2}}
\]
and by Lemma \ref{lem:two5} this gives
\[
N(X,P^{1/2})\ll \frac{QP^{s-1}}{X\log X} + QP^{s-5/4}.
\]
\par
This leaves $N_0$.  By the prime number theorem $D=\exp\big(\vartheta(X)\big)\le P^{1/2}$.  Hence
\[
N_0 = \sum_{d|D}\mu(d) \sum_{m\not=0} R(md^2)=\sum_{d|D}\mu(d) \sum_{m} R(md^2) + O(P^{s-1/2}).
\]
The new inner sum is the number of $\mathbf x$ with $\mathcal P(\mathbf x)\equiv 0\imod{d^2}$ and this is
\[
\left(
\frac{Q}{d^2}+O(1)
\right)\left(
\frac{P}{d^2}+O(1)
\right)^{s-1} \rho_{\mathcal P}(d^2) = QP^{s-1}\frac{\rho_{\mathcal P}(d^2)}{d^{2s}} + O\left(
P^{s-1} \frac{\rho_{\mathcal P}(d^2)}{d^{2s-2}} + \rho_{\mathcal P}(d^2)
\right).
\]
Hence, by Lemma \ref{lem:two5},
\[
N_0 = QP^{s-1}\prod_{p\le X} \left(
1-\frac{\rho_{\mathcal P}(p^2)}{p^{2s}}
\right) + O\left(
P^{s-\frac12+\varepsilon}
\right).
\]
To complete the proof we observe that by Lemma \ref{lem:two5}
\[
\prod_{p> X} \left(
1-\frac{\rho_{\mathcal P}(p^2)}{p^{2s}}
\right) = \exp\left(
-\sum_{p>X} \log\left(
1-\frac{\rho_{\mathcal P}(p^2)}{p^{2s}}
\right)
\right) = 1+O\left(
\frac1{X\log X}
\right).
\]

\end{document}